\documentclass[amsfonts,amsmath,12pt]{article}
\usepackage{amsmath}
\usepackage{amsfonts}
\usepackage{microtype}
\newtheorem{theorem}{Theorem}[section]
\newtheorem{lem}[theorem]{Lemma}
\newtheorem{prop}[theorem]{Proposition}
\newtheorem{cor}[theorem]{Corollary}
\newtheorem{defn}[theorem]{Definition}

\newenvironment{defn-new}{\begin{defn} \em}{\end{defn}}
\newtheorem{rem}[theorem]{Remark}
\newenvironment{rem-new}{\begin{rem} \em}{\end{rem}}
\newtheorem{ex}[theorem]{Example}
\newenvironment{ex-new}{\begin{ex} \em}{\end{ex}}

\newenvironment{notation-new}{\begin{rem} \em}{\end{rem}}

\newenvironment{agr-new}{\begin{rem} \em}{\end{rem}}

\makeatletter \@addtoreset{equation}{section} \makeatother

\makeatletter \@addtoreset{figure}{section} \makeatother

\begin{document}
\begin{center}
{\Large \textbf {Clairaut Generic Riemannian Maps from Nearly K\"ahler
Manifolds}}\\[0pt]

 {\bf  Nidhi Yadav$^{1}$,  Kirti Gupta}\footnote{
Department of Mathematics \& Statistics, Dr. Harisingh Gour Vishwavidyalaya,
Sagar-470 003, M.P. INDIA \newline
Email: nidhiyadav.bina@gmail.com,  guptakirti905@gmail.com}, {\bf Punam Gupta}\footnote{School of Mathematics, Devi Ahilya Vishwavidyalaya, Indore-452 001, M.P. 
INDIA\newline
Email: punam2101@gmail.com} 
\end{center}

\bigskip

\noindent {\bf Abstract.}
In this paper, we study Clairaut generic Riemannian maps from a nearly K\"ahler manifold to a Riemannian manifold. Furthermore, we obtain a condition for a Clairaut  generic Riemannian map to be a totally geodesic foliation on the total manifold. Lastly, we provide non-trivial examples of such Riemannian maps.

\noindent {\bf 2010 Mathematics Subject Classification.}\vspace{0.1cm} 53C12,
53C15, 53C20, 53C55


\section{Introduction}
The idea of a Riemannian map between Riemannian manifolds plays a key role in differential geometry, and in 1992, Riemannian map between two Riemannian manifolds was first introduced by Fischer \cite{Fish} as a generalization of the notion of an isometric immersion and Riemannian submersion. The notions of an immersion and a submersion play a key
 role in the theory of smooth maps between smooth manifolds(finite or
 infinite). If we consider Riemannian manifolds, then in the theory
 of smooth maps between Riemannian manifolds, the two notions of an
 immersion and a submersion correspond to the notions of an isometric immersion and a Riemannian submersion, respectively, and are widely used in differential geometry \cite{Neill},\cite{Watson}.  For the notion of Riemannian maps, we also followed B. Sahin \cite{SahinAR},\cite{SahinSRK}. 

 Recently, B. Sahin \cite{SahinAR} introduced the notion of anti-invariant Riemannian maps, which are Riemannian maps from almost Hermitian
 manifolds to Riemannian manifolds such that the vertical distributions (or, for that matter, the fibers) are anti-invariant under the almost
 complex structure of the total space. Further, as a generalization of
 anti-invariant Riemannian maps, he introduced the notion of conformal
 semi-invariant Riemannian maps when the base manifold is a Riemannian manifold and a K\"ahler manifold \cite{SahinCR},\cite{SahinCSR}. He has shown that such
 maps are very useful for studying the geometry of the total space of
 the Riemannian maps. In the present article, we study the Riemannian
 maps from almost Hermitian manifolds under the assumption that the
 integral manifolds of vertical distributions $\operatorname{ker}F_*$ are generic submanifolds of the total space, and we call this the Riemannian maps with generic fibers. It is not hard to see that one
 can regard it as a generalization of semi-invariant Riemannian maps. Recently, there have been many research papers on the
 geometry of Riemannian maps between various Riemannian manifolds \cite{JaiswalHM,JaiswalNHM,Pandey,ParkSS,ParkSSR,kiran}.  For more study about Clairaut maps, refer \cite{punam,punama,punams}.
 
The paper first reviews the necessary preliminaries on nearly K\"ahler manifolds, O'Neill tensors, and the structure of generic Riemannian maps. It then introduces the notion of a Clairaut generic Riemannian map, characterized by the existence of a girth function analogous to the classical Clairaut relation for geodesics on surfaces of revolution.

This work contains some results that provide necessary and sufficient conditions for a generic Riemannian map to satisfy the Clairaut condition. These conditions are expressed in terms of vertical and horizontal projections, the nearly K\"ahler covariant derivatives, and the O'Neill tensor fields. Several results give criteria for when the distributions $D_1$ and $D_2$ (coming from the decomposition of $\operatorname{ker} F_*$ ) define totally geodesic foliations, generalizing earlier results for invariant, anti-invariant, and slant submersions.

Finally, the paper presents explicit examples of Clairaut generic Riemannian maps from manifolds having nearly-K\"ahler metric.This paper extends Clairaut-type geometry to the setting of generic Riemannian maps on nearly K\"ahler manifolds, enriching the understanding of how complex structures, curvature, and geodesic behavior interact under such maps.

\section{Preliminaries}

An almost complex structure on a smooth manifold $M$ is a smooth tensor
field $J $ of type $(1,1)$ such that $J ^{2}=-I$. A smooth
manifold equipped with such an almost complex structure is called an almost
complex manifold. An almost complex manifold $\left( M,J \right) $
endowed with a chosen Riemannian metric $g$ satisfying
\begin{equation}
g(J X,J Y)=g(X,Y)  \label{eq-ka1}
\end{equation}%
for all $X,Y\in TM$, is called an almost Hermitian manifold.

An almost Hermitian manifold $M$ is called a nearly K\"ahler manifold \cite{Gray} if
\begin{equation}
\left( \nabla _{X}J \right) Y+\left( \nabla _{Y}J \right) X=0
\label{eq-ka2}
\end{equation}%
for all $X,Y\in TM$. If $\left( \nabla _{X}J \right) Y=0$ for all $%
X,Y\in TM$, then $M$ is known as K\"ahler manifold. Every K\"ahler manifold is
nearly K\"ahler but the converse need not be true.

\subsection{Riemannian Maps}
Consider $F: (M,g_{1}) \rightarrow (N,g_{2})$ to be a smooth map between Riemannian manifolds $M$ and $N$ of dimension $m$ and $n,$ respectively, such that $0 < rank F< min \{m,n\},$ so that both 
$\operatorname{ker} F_{*p}$  and $ (\operatorname{range} F_{*p})^\perp$ 
are non-trivial. If $F_{*}:T_{p}M \rightarrow T_{F(p)}N$ denotes the differential map at $p \in M,$ and $F(p) \in N,$ then $T_{p}M$ splits orthogonally with respect to $g_{1}(p)$  as \rm{\cite{Fish}}
\[ T_{p}M= \operatorname{ker}F_{*p} \oplus(\operatorname{ker} F_{*p})^{\perp} = {\mathcal{V}}_{p} \oplus{\mathcal{H}}_{p},\]
where ${\mathcal{V}}_{p}=\operatorname{ker} F_{*p}$ and ${\mathcal{H}}_{p}=(\operatorname {ker} F_{*p})^{\perp}$ are the vertical and horizontal parts of $T_{p}M$, respectively.

$ T_{F(p)}N$ split orthogonally with respect to $g_{2}(F(p))$, therefore $T_{F(p)}N$ can be decomposed as follows: 
$$ T_{F(p)}N= {range F_{*}}_{p} \oplus (range F_{*p})^\perp. $$
Then the map $F: (M,g_{1}) \rightarrow (N,g_{2})$ is called a Riemannian map at $ p \in M,$ if 
\begin{equation}
g_{2}(F_{*}X,F_{*}Y)= g_{1}(X,Y)  \label{f4}
\end{equation}
for all vector fields $X,Y \in \Gamma({\operatorname{ker} F_{*p}})^{\perp}.$
The map $F$ is called a Riemannian map if it satisfies the above condition at each point $p\in M$.

In particular, if $\operatorname{ker} F_*=0$, then a Riemannian map is just an isometric immersion, while if $(range F_*)^\perp=0$,
 then a Riemannian map is nothing but a Riemannian submersion.

The second fundamental tensors of all fibers $F ^{-1}(q),\ q\in N$ give
rise to the tensor field $T$ and $A$ in $M$ defined by O'Neill \cite{Neill} for
arbitrary vector fields $E$ and $F$, which is
\begin{equation}
T_{E}F={\cal H}\nabla _{{\cal V}E}^{M}{\cal V}F+{\cal V}\nabla _{{\cal V}%
E}^{M}{\cal H}F,  \label{EQ2.9}
\end{equation}%
\begin{equation}
A_{E}F={\cal H}\nabla _{{\cal H}E}^{M}{\cal V}F+{\cal V}\nabla _{{\cal H}%
E}^{M}{\cal H}F,  \label{EQ2.8}
\end{equation}%
where ${\cal V}$ and ${\cal H}$ are the vertical and horizontal projections.

On the other hand, from equations (\ref{EQ2.9}) and (\ref{EQ2.8}), we have
\begin{equation}
\nabla _{V}W=T_{V}W+\widehat{\nabla }_{V}W,  \label{EQ2.10}
\end{equation}
\begin{equation}
\nabla _{V}X={\cal H}\nabla _{V}X+T_{V}X,  \label{EQ2.11}
\end{equation}
\begin{equation}
\nabla _{X}V=A_{X}V+{\cal V}\nabla _{X}V,  \label{EQ2.12}
\end{equation}
\begin{equation}
\nabla _{X}Y={\cal H}\nabla _{X}Y+A_{X}Y,  \label{EQ2.13}
\end{equation}
for all $V,W\in \Gamma (\operatorname{ker} F _{\ast })$ and $X,Y\in \Gamma (\operatorname{ker} F
_{\ast })^{\perp },$ where ${\cal V}\nabla _{V}W=\widehat{\nabla }_{V}W.$ 
Recall that a horizontal vector field $X$ on $M$ is called basic if it is $F$-related to a vector field on $N.$ If $X$ is basic, then $A_{X}V={\cal H}\nabla _{V}X.$

\noindent It is easily seen that for $p\in M,$ $U\in {\cal V}_{p}$ and $X\in {\cal H}%
_{p}$ the linear operators
\[
{T}_{U},{A}_{X}:T_{p}M\rightarrow T_{p}M
\]%
are skew-symmetric, that is,
\begin{equation}
g({A}_{X}E,F)=-g(E,{A}_{X}F)\text{ and }g({T}_{U}E,F)=-g(E,{T}_{U}F)
\label{EQ2.14}
\end{equation}%
for all $E,F\in $ $T_{p}M.$ We also see that the restriction of ${T}$ to the
vertical distribution ${T}|_{\operatorname{ker} F _{\ast }\times \operatorname{ker} F _{\ast }}$ is
exactly the second fundamental form of the fibers of $F $. Since ${T}_{U}$
is skew-symmetric, $F $ has totally geodesic fibers if and only
if ${T}\equiv 0$.\newline
In addition, a Riemannian map is a Riemannian map with totally umbilical fibers if \cite{SahinSR}
\begin{equation}\label{eq-2t}
    T_{U}V=g_{1}(U,V)H
\end{equation}
for all $U,V \in \Gamma(\operatorname{ker}F_{*}),$ where $H$ is the mean curvature vector field of fibers.

Let $F:(M,g_{1})\rightarrow (N,g_{2})$ be a smooth map between Riemannian
manifolds. Then the differential $F _{\ast }$ of $F $ can be observed as
a section of the bundle $Hom(TM,F ^{-1}TN)\rightarrow M$, where $F
^{-1}TN$ is the bundle that has fibers $\left( F ^{-1}TN\right)
_{x}=T_{f(x)}N$, has a connection $\nabla $ induced from the Riemannian
connection $\nabla ^{M}$, and the pullback connection $\stackrel{F}{\nabla}.$  Then the second
fundamental form of $F $ is given by
\begin{equation}
(\nabla F _{\ast })(X,Y)=\stackrel{F}{\nabla} _{X}F _{\ast }Y-F _{\ast }(\stackrel{M}{\nabla}
_{X}Y)\text{ \ for all} \quad X,Y\in \Gamma (TM),  \label{EQ2.15}
\end{equation}%
where $\nabla ^{N}$ is the pullback connection \cite{Baird}. It is known that the second fundamental form is symmetric. In \cite{SahinAR}, \c Sahin proved that $(\nabla F_{*})(X,Y)$ has no component in $\operatorname{range}F_{*},$ for all $X,Y \in \Gamma(\operatorname{ker}F_{*})^{\perp}.$ More Precisely, we have 
\begin{equation}
    (\nabla{F}_{*})(X,Y)\in \Gamma (\operatorname{range}F_{*})^{\perp}.
    \label{f1}
\end{equation}
We also know
that $F $ is said to be a totally geodesic map \cite{Baird} if $(\nabla F
_{\ast })(X,Y)=0$ for all $X,Y\in TM$.\\

Let $F$ be a Riemannian map from an almost Hermitian manifold $\left(M, g_{1}, J\right)$ to a Riemannian manifold $(N, g_{2}).$ Define

$$
{\mathcal{D}}_p=\left(\operatorname{ker} F_{* p} \cap J\left(\operatorname{ker} F_{* p}\right)\right), \quad p \in M
$$
the complex subspace of the vertical subspace ${\mathcal{V}}_p.$
\begin{defn-new}\cite{Sahingen}
 Let $F$ be a Riemannian map from an almost Hermitian manifold $(M, g_{1}, J)$ to a Riemannian manifold $(N, g_{2}).$ If the dimension $\mathcal{D}_p$ is constant along $M$ and it defines a differentiable distribution on $M$ then we say that $F$ is a generic Riemannian map.   
\end{defn-new} 

A generic Riemannian map is purely real if ${\mathcal{D}}_p=\{0\}$ and complex if ${\mathcal{D}}_p= \operatorname{ker}F_{*p}$. 

\noindent Let $\mathcal{D}_1$ denote the differentiable distribution determined by $\mathcal{D}_p.$ For a generic Riemannian map, the orthogonal complementary distribution $\mathcal{D}_2$ of $\mathcal{D}_1$, called the purely real distribution, satisfies

\begin{equation}\label{eq-1d}
\operatorname{ker} F_*={\mathcal{D}}_{1} \oplus {\mathcal{D}}_{2},     
\end{equation}
and
$$
{\mathcal{D}}_{1} \cap {\mathcal{D}}_{2} =\{0\}.
$$
Let $F$ be a generic Riemannian map from an almost Hermitian manifold $(M, g_{1}, J)$  to a Riemannian manifold $(N, g_{2}).$ Then for $U \in \Gamma\left(\operatorname{ker} F_*\right)$, we write
\begin{equation} \label{eq-1v}
 J U=\phi U+\omega U,   
\end{equation}
where $\phi U \in \Gamma\left(\operatorname{ker} F_*\right)$ and $\omega U \in \Gamma\left(\left(\operatorname{ker} F_*\right)^{\perp}\right)$. 
Now we consider the complementary orthogonal distribution $\mu$ to $\omega {\mathcal{D}}_{2}$ in $\left(\operatorname{ker} F_*\right)^{\perp}$. It is obvious that we have
$$
\phi {\mathcal{D}}_{2} \subseteq {\mathcal{D}}_{2}, \quad \left(\operatorname{ker} F_*\right)^{\perp}=\omega {\mathcal{D}}_{2} \oplus \mu .
$$
Also for $V,W \in \Gamma (\operatorname{ker} F _{\ast }),$ we have 
\begin{equation} \label{eq-abc}
({\nabla}_{V}\phi)W=\widehat{{\nabla}}_{V}\phi W- \phi \widehat{{\nabla}}_{V}W,
\end{equation}
\begin{equation} \label{eq-ab}
    ({\nabla}_{V}\omega)W= {\mathcal{H}}{\nabla}_{V}\omega W- \omega \widehat{{\nabla}}_{V}W.
\end{equation}
For $X \in \Gamma\left(\left(\operatorname{ker} F_*\right)^{\perp}\right)$, we write
\begin{equation}\label{eq-2h}
J X=BX+CX,   
\end{equation}
where $ BX \in \Gamma\left({\mathcal{D}}_{2}\right)$ and $CX \in \Gamma(\mu)$. Then it is clear that we get
$$
B\left(\left(\operatorname{ker} F_*\right)^{\perp}\right)={\mathcal{D}}_{2}.
$$
Considering \rm{(\ref{eq-1d})}, for $U \in \Gamma\left(\operatorname{ker} F_*\right)$, we can write

\begin{equation}
J U=P_1 U+P_2 U+\omega U,    \label{h6}
\end{equation}
where $P_1$ and $P_2$ are the projections from $\operatorname{ker} F_*$ to ${\mathcal{D}}_{1}$ and ${\mathcal{D}}_{2}$, respectively.\\

Let $F $ be a generic Riemannian map from a nearly K\"ahler
manifold $(M,J ,g_{1})$ onto Riemannian manifolds $(N,g_{2})$. For any
arbitrary tangent vector fields $U$ and $V$ on $M$, we set
\begin{equation}
(\nabla _{U}J )V=P_{U}V+Q_{U}V, \label{eq-c1}
\end{equation}%
where $P_{U}V, Q_{U}V$ denotes the horizontal and vertical parts of $(\nabla
_{U}J )V$, respectively. Clearly, if $M$ is a K\"ahler manifold, then $%
P=Q=0$.\newline
If $M$ is a nearly K\"ahler manifold, then $P$ and $Q$ satisfy
\begin{equation}
P_{U}V=-P_{V}U,\qquad Q_{U}V=-Q_{V}U.  \label{eq-c2}
\end{equation}

\section{Clairaut Generic Riemannian Map}
Let $S$ be a revolution surface in ${\mathbb R}^{3}$ with rotation axis $L$.
For any $p\in S$, we denote by $r(p)$ the distance from $p$ to $L$. Given a
geodesic $\alpha:K\subset {\mathbb R}\rightarrow S$ on $S$, let $\theta (t)$
be the angle between $\alpha (t)$ and the meridian curve through $h
(t), t\in I$. A well-known Clairaut's theorem states that for any geodesic on $%
S $, the product $r\sin \theta $ is constant along $\alpha $, i.e., it is
independent of $t$. In the theory of Riemannian submersions, Bishop \cite%
{Bishop} introduces the notion of Clairaut Riemannian submersion, and the notion of Clairaut Riemannian map was defined by \c Sahin \cite{SahinCC} in the following way

\begin{defn-new}
\cite{SahinCC} 
A Riemannian map  $F :(M,g_{1})\rightarrow (N,g_{2})$ is
called a Clairaut Riemannian map if there exists a positive function $\tilde{r}$ on $M$,
which is known as the girth of the Riemannian map, such that, for every geodesic $%
\alpha$ on $M$, the function $(\tilde{r}\circ \alpha )\sin \theta $ is constant,
where$\ \theta (t)$ is the angle between $\dot{\alpha}(t)$ and the
horizontal space at $\alpha (t)$,  for any $t$.
\end{defn-new}

He also gave the following necessary and sufficient condition for a
Riemannian map to be a Clairaut Riemannian map: 

\begin{theorem}
\label{th-bis} \rm{\cite{Bishop}} Let $F :(M,g_{1})\rightarrow (N,g_{2})$ be a
Riemannian map with connected fibers. Then, $F $ is a Clairaut
Riemannian map with $\tilde{r}=e^{f}$ if and only if each fiber is totally umbilical and
has the mean curvature vector field $H=-{\operatorname{grad}}f$, where ${\operatorname{grad}}f$ is the
gradient of the function $f$ with respect to $g$.
\end{theorem}

\begin{defn-new}
 A generic Riemannian map from nearly K\"ahler manifold to Riemannian manifold is called Clairaut generic Riemannian map if it satisfies the condition of Clairaut Riemannian map.   
\end{defn-new}
Next, we prove some results and reveal some new structural behavior of Clairaut generic Riemannian maps, especially in nearly-K\"ahler settings.
\begin{lem}
Let $F$ be a generic Riemannian map from a nearly K\"ahler manifold $(M, g_1, J)$ to a Riemannian manifold  $(N, g_2).$  If $\alpha: I \rightarrow M$ is a regular curve and $X(t), U(t)$ denotes the horizontal and vertical components of its tangent vector field, then $\alpha$ is a geodesic on $M$ if and only if

\begin{align}
& {\mathcal{V}} \nabla_{\dot{\alpha}} B X+A_X C X+ {\mathcal{V}} \nabla_{\dot{\alpha}}\phi U+ A_X \omega U+T_U CX + T_U \omega U=0, \label{f8}\\
& {\mathcal{H}} \nabla_{\dot{\alpha}} C X+{\mathcal{H}} \nabla_{\dot{\alpha}} \omega U+A_X B X+T_U B X + A_X \phi U + T_U\phi U=0 .\label{f9}
\end{align}

\end{lem}
\noindent
{\bf{Proof.}} Let $\alpha: I \rightarrow M$ be a regular curve on $M$. Since $J ^{2}\dot{\alpha}=-\dot{\alpha}$. Taking the covariant
derivative of this, 
we have
\begin{equation}
\left( \nabla _{\dot{\alpha}}J \right) J \dot{\alpha}+J \left(
\nabla _{\dot{\alpha}}J \dot{\alpha}\right) =-\nabla _{\dot{\alpha}}\dot{\alpha}.
\label{eq-1}
\end{equation}%
Since $U(t)$ and $X(t)$ are the vertical and horizontal parts of the tangent
vector field $\dot{\alpha}(t)=W$ of $\alpha(t)$, that is, $\dot{\alpha}=U+X$. So (\ref{eq-1}%
) becomes
\begin{eqnarray}
-\nabla _{\dot{\alpha}}\dot{\alpha} &=&J \left( \nabla _{U+X}J
(U+X)\right) +P_{\dot{\alpha}}J \dot{\alpha}+Q_{\dot{\alpha}}J \dot{\alpha}
\nonumber \\
&=&J \left( \nabla _{U}J U+\nabla _{X}J U+\nabla
_{U}J X+\nabla _{X}J X\right) +P_{\dot{\alpha}}J \dot{\alpha}+Q_{%
\dot{\alpha}}J \dot{\alpha}  \nonumber \\
&=&J \left( \nabla _{U}(\phi U+ \omega U)+\nabla _{X}(\phi U+ \omega U)+\nabla
_{U}\left( B X+ C X\right) +\nabla _{X}\left( B X+ C
X\right) \right)   \nonumber \\
&&+P_{\dot{\alpha}}J \dot{\alpha}+Q_{\dot{\alpha}}J \dot{\alpha}.  \label{eq-2}
\end{eqnarray}%
Using (\ref{EQ2.10})-(\ref{EQ2.13}) in (\ref{eq-2}), we get
\begin{eqnarray}
-\nabla _{\dot{\alpha}}\dot{\alpha} &=&J \left( {\cal H}\left( \nabla _{\dot{\alpha}%
}\omega U+\nabla _{\dot{\alpha}}C X\right) +A_{X}B X+A_{X}C
X+A_{X}\omega U +A_{X}\phi U + {\cal V}\nabla_{X}\phi U\right.   \nonumber \\
&&\left. +T_{U}C X+T_{U}B X+{\cal V}\nabla _{X}B
X+T_{U}\omega U+\widehat{\nabla}_{U}B X +T_U \phi U+\widehat{\nabla}_{U}\phi U \right)\nonumber \\
&& +P_{\dot{\alpha}}J \dot{\alpha}%
+Q_{\dot{\alpha}}J \dot{\alpha}.  \label{eq-3}
\end{eqnarray}%
Let $Y,Z\in TM$. Since $J ^{2}Z=-Z$, on differentiation, we have
\[
J \left( \nabla _{Y}J Z\right) +\left( \nabla _{Y}J
\right) J Z=-\nabla _{Y}Z,
\]%
\[
J ^{2}\left( \nabla _{Y}Z\right) +J \left( \nabla _{Y}J
\right) Z+\left( \nabla _{Y}J \right) J Z=-\nabla _{Y}Z,
\]%
using (\ref{eq-c1}) in above, we obtain
\begin{equation}
J \left( P_{Y}Z+Q_{Y}Z\right) =-P_{Y}J Z-Q_{Y}J Z.
\label{eq-c3}
\end{equation}%
By (\ref{eq-c3}), we have
\[
J \left( P_{\dot{\alpha}}J \dot{\alpha}+Q_{\dot{\alpha}}J \dot{\alpha}\right)
=P_{\dot{\alpha}}\dot{\alpha}+Q_{\dot{\alpha}}\dot{\alpha},
\]%
since $P$ and $Q$ are antisymmetric, so
\begin{equation}
J \left( P_{\dot{\alpha}}J \dot{\alpha}+Q_{\dot{\alpha}}J \dot{\alpha}\right)
=0.  \label{eq-c4}
\end{equation}%
Using (\ref{eq-c4}) and equating the vertical and horizontal part of (\ref%
{eq-3}), we obtain

$$
\begin{aligned}
{\mathcal{V}} J \nabla_{\dot{\alpha}} \dot{\alpha} & ={\mathcal{V}} \nabla_{\dot{\alpha}} B X+A_X C X+ {\mathcal{V}} \nabla_{\dot{\alpha}}\phi U+ A_X \omega U+T_U CX + T_U \omega U, \\
{\mathcal{H}} J \nabla_{\dot{\alpha}} \dot{\alpha} & ={\mathcal{H}} \nabla_{\dot{\alpha}} C X+{\mathcal{H}} \nabla_{\dot{\alpha}} \omega U+A_X B X+T_U B X + A_X \phi U + T_U\phi U .
\end{aligned}
$$
Now, $\alpha$ is a geodesic on $M$ if and only if ${\mathcal{V}} J \nabla_{\dot{\alpha}} \dot{\alpha}=0$ and ${\mathcal{H}} J \nabla_{\dot{\alpha}} \dot{\alpha}=0$, which completes the proof.

\begin{theorem} 
Let $F$ be a generic Riemannian map from a nearly K\"ahler manifold $(M, g_1, J)$ to a Riemannian manifold  $(N, g_2).$ Then, $F$ is a Clairaut generic Riemannian map with $\tilde{r}=e^f$ if and only if 
$$\begin{aligned}
    &g_{1 \alpha(t)}\left( {\mathcal{V}} \nabla_{\dot{\alpha}}\phi U +A_X C X+\right. \left.T_U C X+\left(A_X+T_U\right) \omega U, B X\right)\\
  &  +g_{1 \alpha(t)}\left(A_X B X +\left(A_X+T_U\right) \phi U+T_U B X+{\mathcal{H}} \nabla_{\dot{\alpha}} \omega U, C X\right)+ g_{1 \alpha(t)}(U, U) \frac{d f}{d t}=0,
\end{aligned}
$$
where $\alpha: I \rightarrow M$ is a geodesic on $M$ and $X, U$ are horizontal and vertical components of $\dot{\alpha}(t)$.
\end{theorem}
\noindent {\bf Proof:} Let $\alpha: I \rightarrow M$ be a geodesic on $M$ with $U(t)={\mathcal{V}} \dot{\alpha}(t)$ and $X(t)= {\mathcal{H}} \dot{\alpha}(t), \theta(t)$ denote the angle in {\rm$[ 0,\pi ]$} between $\dot{\alpha}(t)$ and $X(t)$. Assuming $a= \|\dot{\alpha}(t)\|^2$, we get
\begin{align}
& g_{1 \alpha(t)}(X(t), X(t))=a \cos ^2 \theta(t), \label{g1}\\
& g_{1 \alpha(t)}(U(t), U(t))=a \sin ^2 \theta(t).  \label{g2}
\end{align}
Differentiating (\ref{g1} ), we get
\begin{equation}
\frac{\mathrm{d}}{\mathrm{~d} t} g_{1 \alpha(t)}(X(t), X(t))= -2 a \cos \theta \sin \theta \frac{\mathrm{~d} \theta}{\mathrm{~d} t} .\label{f5}
\end{equation}
On the other hand, using (\ref{eq-ka1}), we get
\begin{equation}
\frac{\mathrm{d}}{\mathrm{~d} t} g_{1 \alpha(t)}(X, X)= \frac{\mathrm{d}}{\mathrm{~d} t} g_{1 \alpha(t)}(J X, J X) .
\label{g3}
\end{equation}
Since $F$ is generic Riemannian map, using (\ref{eq-2h}) in (\ref{g3}), we get
\begin{align}
\frac{\mathrm{d}}{\mathrm{~d} t} g_{1 \alpha(t)}(X(t), X(t))&= 2 g_{1 \alpha(t)}\left(\nabla_{\dot{\alpha}} B X, B X\right) \notag\\
& +2 g_2\left(F_*\left(\nabla_{\dot{\alpha}} C X\right), F_*(C X)\right) .
\label{g4}
\end{align}
Using (\ref{EQ2.15}) in (\ref{g4}), we obtain
$$
\begin{aligned}
\frac{\mathrm{d}}{\mathrm{~d} t} g_{1 \alpha(t)}(X(t), X(t))&= 2 g_{1 \alpha(t)}\left(\nabla_{\dot{\alpha}} B X, B X\right)+2 g_2\left(-\left(\nabla F_*\right)(\dot{\alpha}, C X)\right. \\
& \left.+\stackrel{F}\nabla_{\dot{\alpha}} F_*(C X), F_*(C X)\right).
\end{aligned}
$$
Since the second fundamental form of $F$ is linear, we obtain the following from the above equation:
\begin{align}
\frac{\mathrm{d}}{\mathrm{~d} t} g_{1 \alpha(t)}(X(t), X(t))&= 2 g_{1 \alpha(t)}\left(\nabla_{\dot{\alpha}} B X, B X\right)  +2 g_2\left(-\left(\nabla F_*\right)(U, C X)\right ) \notag \\
& -\left(\nabla F_*\right)(X, C X)  \left.+\nabla_{X+U}^F F_*(C X), F_*(C X)\right).
\label{f2}
\end{align}
In addition, from (\ref{EQ2.15}), (\ref{f1}) and (\ref{f2}), we get
\begin{align}
\frac{\mathrm{d}}{\mathrm{~d} t} g_{1 \alpha(t)}(X(t), X(t))&= 2 g_{1 \alpha(t)}\left(\nabla_{\dot{\alpha}} B X, B X\right)+2 g_2\left(-\nabla_U F_*(C X)\right.\notag \\
& +F_*\left(\nabla_U C X\right)+\nabla_X F_*(C X) \notag \\
& \left.+\nabla_U F_*(C X), F_*(C X)\right) .\label{f3}
\end{align}
Using (\ref{f4}) and (\ref{EQ2.15}) in (\ref{f3}), we obtain
\begin{align}
\frac{\mathrm{d}}{\mathrm{~d} t} g_{1 \alpha(t)}(X(t), X(t))&= 2 g_{1 \alpha(t)}\left(\nabla_{\dot{\alpha}} B X, B X\right)  +2 g_{1 \alpha(t)}\left({\mathcal{H}} \nabla_{\dot{\alpha}} C X, C X\right) .
\label{f6}
\end{align}
Now, from (\ref{f5}) and (\ref{f6}), we get
\begin{equation}
g_{1 \alpha(t)}\left(\nabla_{\dot{\alpha}} B X, B X\right)+g_{1 \alpha(t)}\left({\mathcal{H}} \nabla_{\dot{\alpha}} C X, C X\right)= -a \cos \theta \sin \theta \frac{\mathrm{~d} \theta}{\mathrm{~d} t}.\label{f7}
\end{equation}
Using (\ref{f8}) and (\ref{f9}) in (\ref{f7}), we get
\begin{align}
 &g_{1 \alpha(t)}\left( {\mathcal{V}} \nabla_{\dot{\alpha}}\phi U +A_X C X+T_U C X+\left(A_X+T_U\right) \omega U, B X\right) \notag \\
  &  +g_{1 \alpha(t)}\left(A_X B X +\left(A_X+T_U\right) \phi U+T_U B X+{\mathcal{H}} \nabla_{\dot{\alpha}} \omega U, C X\right)= a \cos \theta \sin \theta \frac{\mathrm{~d} \theta}{\mathrm{~d} t}. \label{f10}
\end{align}
Moreover, $F$ is a Clairaut Riemannian map with $\tilde{r}=e^f$ if and only if $\frac{\mathrm{d}}{\mathrm{d} t}\left(e^{f \circ \alpha} \sin \theta\right)=0$, that is, $e^{f \circ \alpha}\left(\cos \theta \frac{\mathrm{~d} \theta}{\mathrm{~d} t}+\sin \theta \frac{\mathrm{d} f}{\mathrm{~d} t}\right)=0$. Multiplying this by a non-zero factor $a \sin \theta$, we get
\begin{equation}
-a \cos \theta \sin \theta \frac{\mathrm{~d} \theta}{\mathrm{~d} t}=a \sin ^2 \theta \frac{\mathrm{~d} f}{\mathrm{~d} t}.\label{f11}
\end{equation}
Thus, from (\ref{g2}), (\ref{f10}) and (\ref{f11}), we get
$$
\begin{aligned}
 &g_{1 \alpha(t)}\left( {\mathcal{V}} \nabla_{\dot{\alpha}}\phi U +A_X C X+\right. \left.T_U C X+\left(A_X+T_U\right) \omega U, B X\right)\\
  &  +g_{1 \alpha(t)}\left(A_X B X +\left(A_X+T_U\right) \phi U+T_U B X+{\mathcal{H}} \nabla_{\dot{\alpha}} \omega U, C X\right)=-g_{1 \alpha(t)}(U, U) \frac{\mathrm{d}f}{\mathrm{~d} t},
\end{aligned}
$$
which completes the proof.
\begin{theorem}\label{th-R}
    Let $F$ be a Clairaut generic Riemannian map from a nearly K\"ahler manifold $(M,g_{1},J)$ to a Riemannian manifold $(N,g_{2})$ with $\tilde{r}=e^{f},$. Then at least one of the following statements is true:
    \begin{itemize}
        \item[\rm(i)] $f$ is constant on $\omega{\mathcal{D}}_{2}.$
       
        \item[\rm(ii)] For all $Y \in \Gamma(\operatorname{ker} F_*)^\perp$ and $V,W \in \Gamma\left({\mathcal{D}}_{2}\right)$
        \begin{equation}\label{eq-y1}
    g_1(V, W) C\operatorname{grad} f=-{\mathcal H}\nabla _V CW.
\end{equation}
    \end{itemize}
\end{theorem}
{\bf Proof:} Let $F$ be a Clairaut generic Riemannian map from a nearly K\"ahler manifold to a Riemannian manifold. Then, using (\ref{eq-2t}) in Theorem \ref{th-bis}, we get
\begin{equation}
   T_U V=-g_1(U, V) \operatorname{grad} f \label{f12}
\end{equation}
for all $U, V \in \Gamma\left({\mathcal{D}}_{2}\right)$, which implies
\begin{equation}
    g_1\left(T_U V, \omega W\right)=-g_1(U, V) g_1(\operatorname{grad} f, \omega W)\label{f13}
\end{equation}
for all $W \in \Gamma\left({\mathcal D}_2\right)$. Now, from (\ref{eq-ka1}), (\ref{EQ2.10}) and (\ref{f13}), we get
\begin{equation}
    g_1\left(\nabla_U  V, \omega W\right)=-g_1(U, V) g_1(\operatorname{grad} f,\omega W) .\label{f14}
\end{equation}
Now
\begin{align}
g_1(\nabla_U V, \omega W) &= -g( V,  \nabla_U\omega W)\notag \\
&= -g_1\left(V, \omega \nabla_U W \right)-g_1\left(V,(\nabla_U\omega)  W \right)\notag\\
&= g_1\left(  \omega V,  \nabla_U W \right) -g_1\left(V,(\nabla_U\omega)  W \right)\notag\\
&=g_1\left(  \omega V,  T_U W \right) -g_1\left(V,(\nabla_U\omega)  W \right).
\label{f14.}
\end{align}
\noindent Expanding via the O'Neill tensor components or sub-distribution properties and using (\ref{f14}), (\ref{f14.}):
\begin{align}
g_1(U, V) g_1(\operatorname{grad} f,\omega W) = g_1(U, W) g_1(\operatorname{grad} f,\omega V)+g_1\left(V,(\nabla_U\omega)  W \right).
\label{g1.}
\end{align}
Taking $U=W$ and interchanging the roles of $U$ and $V$, we obtain \begin{equation}
g_1(V, V) g_1(\operatorname{grad} f,\omega U) = g_1(U, V) g_1(\operatorname{grad} f,\omega V).\label{f14..}\end{equation}
Using (\ref{g1.}) with $W=U$ in  (\ref{f14..}), we have 
\begin{equation}
g_1(\operatorname{grad} f,\omega U) =\dfrac{(g_1(U, V))^2}{g_1(U, U)g_1(V, V)}  g_1(\operatorname{grad} f,\omega U).\label{f14...}\end{equation}
If $\operatorname{grad} f \in \Gamma\left(\omega {\mathcal{D}}_{2}\right)$ and also considering the equality case of the Schwarz inequality, then $f$ is constant on $\omega {\mathcal{D}}_{2}$.\\


For the second part, now, from (\ref{EQ2.10}) and (\ref{f12}), we get
\begin{equation}\label{eq-z}
    g_1\left(\nabla_V W, CY\right)=-g_1(V, W) g_1(\operatorname{grad} f, CY)
\end{equation}
for all $Y \in \Gamma(\operatorname{ker} F_*)^\perp$. By Leibniz rule, we get
\begin{equation}
    g_1\left(\nabla_V W, CY\right)=-g_1({\mathcal H}\nabla _V CW+T_V CW, Y)
\end{equation}
which implies
\begin{equation}\label{eq-y}
    g_1(V, W) g_1(\operatorname{grad} f, CY)=g_1({\mathcal H}\nabla _V CW, Y).
\end{equation}

\begin{cor}
 Let $F$ be a Clairaut generic Riemannian map from a nearly K\"ahler manifold $(M,g_{1},J)$ to a Riemannian manifold $(N,g_{2})$ with $\tilde{r}=e^{f}$. Then the fibers of $F$ are totally geodesic if and only if ${F^*}{\nabla}_{V} CW=0$  for all  $V,W \in \Gamma\left({\mathcal{D}}_{2}\right).$
\end{cor}
 {\bf Proof:} From Theorem \rm{\ref{th-R}}, if the fibers are totally geodesic, then the mean curvature vector vanishes, and hence $\operatorname{grad}f=0$. Therefore, equation (\ref{eq-y}) reduces to
\[
 {F^*}{\nabla}_{V} CW=0.
\]

\begin{prop}
Let $F$ be a Clairaut generic Riemannian map from a nearly K\"ahler manifold $(M,g_{1},J)$ to a Riemannian manifold $(N,g_{2})$ with $\tilde{r}=e^{f}$ and $dim({\mathcal{V}})> 1.$ Then ${\mathcal{D}}_{1}$ defines a totally geodesic foliation on $M$ if and only if $$g_{1} ( Q_X Y- {\widehat{\nabla}}_{X}\phi Y -T_X \omega Y , \phi Z) + g_2( F_*(P_{X}Y - \nabla_{X}\omega Y), F_*( \omega Z) )=0$$ for $X,Y\in \Gamma({\mathcal{D}}_{1}) $ and $Z \in \Gamma({\mathcal{D}}_{2}).$     
\end{prop}
{\bf Proof:} For $X,Y \in \Gamma({\mathcal{D}}_{1})$ and $Z \in \Gamma({\mathcal{D}}_{2})$ using \rm(\ref{eq-ka1}), \rm(\ref{eq-1v}) and \rm(\ref{eq-c1}), we have
$$
\begin{aligned}
    g_{1}(\nabla_{X}Y,Z)
    &= g_{1}(\nabla_{X}JY,JZ)-g_{1}((\nabla_{X}J)Y,JZ)\\
    &= g_{1}(\nabla_{X}(\phi Y + \omega Y),JZ)-g_{1}(P_{X}Y+Q_{X}Y,JZ)\\
    &= g_{1}(\nabla_{X}\phi Y,JZ)+g_{1}( JZ,\nabla_X \omega Y)-g_{1}(P_{X}Y,JZ)- g_1(Q_X Y, JZ)\\
    &= g_{1}(T_X \phi Y+ {\widehat{\nabla}}_{X}\phi Y, JZ) + g_{1}({\mathcal{H}} \nabla_{X}\omega Y +  T_X \omega Y, JZ) - g_{1}(P_{X}Y,\omega Z)- g_1(Q_X Y, \phi Z).
\end{aligned}
$$
Since $\mathcal{D}_1$ defines a totally geodesic foliation on $M$, if and only if \
$g_1(\nabla_XY,Z)=0 $  for all $X,Y\in \Gamma(\mathcal{D}_1)$ and $Z\in \Gamma(\mathcal{D}_2)$,

\vspace{.2cm}
\noindent
Using this together with \rm(\ref{EQ2.10}), we obtain 
$$
\begin{aligned}
g_{1}(T_X \phi Y+ {\widehat{\nabla}}_{X}\phi Y, JZ) &= - g_{1}({\mathcal{H}} \nabla_{X}\omega Y +  T_X \omega Y, JZ) + g_{1}(P_{X}Y,\omega Z)+ g_1(Q_X Y, \phi Z),\\
g_{1}(T_X \phi Y, \omega Z) &= -g_{1} ({\widehat{\nabla}}_{X}\phi Y, \phi Z) - g_{1}({\mathcal{H}} \nabla_{X}\omega Y , \omega Z )- g_1(  T_X \omega Y, \phi Z)\\
& + g_{1}(P_{X}Y,\omega Z)+ g_1(Q_X Y, \phi Z),\\
-g_{1}(X,\phi Y)g_{1}(\operatorname{grad}f, \omega Z)&= g_{1} ( Q_X Y- {\widehat{\nabla}}_{X}\phi Y -T_X \omega Y , \phi Z) + g_1( P_{X}Y - \nabla_{X}\omega Y,  \omega Z ).
\end{aligned}
$$
Since $F$ is a Clairaut Riemannian map, from \rm(\ref{f12}), we get
 \begin{align}
 -g_{1}(X,\phi Y)g_{1}(\operatorname{grad}f, \omega Z)= g_{1} ( Q_X Y- {\widehat{\nabla}}_{X}\phi Y -T_X \omega Y , \phi Z) + g_2( F_*(P_{X}Y - \nabla_{X}\omega Y), F_*( \omega Z) ).
 \end{align}  

Since $dim({\mathcal{V}}) > 1,$ so $f$ is constant on $\omega {\mathcal{D}}_{2}.$ Therefore, from above equation, we get $g_{1} ( Q_X Y- {\widehat{\nabla}}_{X}\phi Y -T_X \omega Y , \phi Z) + g_2( F_*(P_{X}Y - \nabla_{X}\omega Y), F_*( \omega Z) )=0,$
which proves the assertion.

\begin{prop}
   Let $F$ be a Clairaut generic Riemannian map from a nearly K\"ahler manifold $(M,g_{1},J)$ to a Riemannian manifold $(N,g_{2})$ with $\tilde{r}=e^{f}.$ Then, ${\mathcal{D}}_{2}$ defines a totally geodesic foliation on $M$ if and only if
   
   $$ g_{2}(F_{\ast}(P_{X}Z-\nabla_{X}\omega Z),F_{\ast}(\omega Y)) = g_{1}(\widehat{\nabla}_{X}\phi Z+T_{X}\omega Z -Q_X Z, \phi Y) $$
for $X,Y \in \Gamma({\mathcal{D}}_{2})$ and $Z \in \Gamma({\mathcal{D}}_{1})$.
   \end{prop}

\noindent 
{\bf Proof:} 
 For $X,Y \in \Gamma({\mathcal{D}}_{2})$ and $Z \in \Gamma({\mathcal{D}}_{1})$ using \rm(\ref{eq-ka1}), \rm(\ref{eq-1v}) and \rm(\ref{eq-c1}), we have
$$
\begin{aligned}
    g_{1}(\nabla_{X}Y,Z)
    &= -g_{1}(\nabla_{X}JZ,JY)+g_{1}((\nabla_{X}J)Z,JY)\\
    &= -g_{1}(\nabla_{X}(\phi Z + \omega Z),JY)+g_{1}(P_{X}Z+Q_{X}Z,JY)\\
    &= -g_{1}(\nabla_{X}\phi Z,JY)-g_{1}(\nabla_{X}\omega Z,JY)+g_{1}(P_{X}Z,JY)+g_{1}(Q_{X}Z,JY),
\end{aligned}
$$
since ${\mathcal{D}}_{2}$ defines a totally geodesic foliation on $M,$ therefore, using \rm(\ref{EQ2.10}), \rm(\ref{EQ2.11}) and \rm(\ref{f12}), we obtain
$$
\begin{aligned}
  g_{1}(P_{X}Z,\omega Y)+g_{1}(Q_{X}Z,\phi Y)&=g_{1}(T_{X}\phi Z+ {\widehat{\nabla}}_{X}\phi Z,\phi Y+ \omega Y)+ g_{1}({\mathcal{H}}\nabla_{X}\omega Z + T_{X}\omega Z,\phi Y + \omega Y)\\
  & = g_{1}(T_{X}\phi Z, \omega Y)+ g_{1}( {\widehat{\nabla}}_{X}\phi Z,\phi Y)+ g_{1}({\mathcal{H}}\nabla_{X}\omega Z, \omega Y)+g_{1}(T_{X}\omega Z,\phi Y ). 
\end{aligned}
$$
$$
\begin{aligned}
    g_{2}(F_{\ast}(P_{X}Z),F_{\ast}(\omega Y))+g_{1}(Q_{X}Z,\phi Y)=& -g_{1}(X, \phi Z)g_{1}(\operatorname{grad}f, \omega Y)+ g_{1}(\widehat{\nabla}_{X}\phi Z, \phi Y)\\&+ g_{2}(F_{\ast}(\nabla_{X}\omega Z), F_{\ast}(\omega Y))+g_{1}(T_{X}\omega Z, \phi Y).
\end{aligned}
$$
From (\ref{h6}), we have
$$
\begin{aligned}
    g_{2}(F_{\ast}(P_{X}Z),F_{\ast}(\omega Y))-g_{2}(F_{\ast}(\nabla_{X}\omega Z), F_{\ast}(\omega Y))=& -g_{1}(X, P_{1} Z)g_{1}(\operatorname{grad}f, \omega Y)+g_{1}(\widehat{\nabla}_{X}\phi Z, \phi Y) \\&+g_{1}(T_{X}\omega Z, \phi Y)- g_{1}(Q_{X}Z,\phi Y).
\end{aligned}
$$
$$
\begin{aligned}
 g_{2}(F_{\ast}(P_{X}Z-\nabla_{X}\omega Z),F_{\ast}(\omega Y)) =&  -g_{1}(X, P_{1} Z)g_{1}(\operatorname{grad}f, \omega Y)+g_{1}(\widehat{\nabla}_{X}\phi Z+T_{X}\omega Z -Q_X Z, \phi Y).
\end{aligned}
$$
Since $X \in \Gamma({\mathcal{D}}_{2})$ and $P_1 Z \in \Gamma({\mathcal{D}}_{1})$ are orthogonal to each other. Therefore, from the above equation, we get 
$$
\begin{aligned}
 g_{2}(F_{\ast}(P_{X}Z-\nabla_{X}\omega Z),F_{\ast}(\omega Y)) =& g_{1}(\widehat{\nabla}_{X}\phi Z+T_{X}\omega Z -Q_X Z, \phi Y),
\end{aligned}
$$
 which completes the proof.
 \begin{ex-new}
     Let $\left({\mathbb{R}}^{10},J, g_1\right)$ be a nearly K\"ahler manifold endowed with the Euclidean metric $g_1$ on ${\mathbb{R}}^{10}$ given by
$$
g_1=\sum_{i=1}^{10} d x_i^2,
$$
\noindent
canonical complex structure
$$
J\left(x_1, x_2, x_3, x_4, x_5, x_6,x_7, x_8, x_9, x_{10}\right)=\left( -x_2, x_1, -x_4, x_3, -x_6,x_5, -x_8, x_7, -x_{10}, x_9\right)
$$
and the $J$-basis is $\left\{\left.e_i=\dfrac{\partial}{\partial x_i} \right\rvert\, i=1, \ldots, 10\right\}$. Let $\left({\mathbb{R}}^7, g_2\right)$ be a Riemannian manifold endowed with the metric $g_2=\displaystyle\sum_{i=1}^7 d y_i^2$.

Consider a map $F:\left({\mathbb{R}}^{10},J, g_1\right) \rightarrow\left({\mathbb{R}}^7, g_2\right)$ defined by

$$
F\left(x_1, x_2, x_3, x_4,x_5,x_6,x_7,x_8,x_9,x_{10}\right)=\left(x_2,x_1,\frac{x_3+x_4}{\sqrt{2}}, 0,0,\frac{x_8+x_{10}}{\sqrt{2}}, \frac{x_7+x_{9}}{\sqrt{2}}\right) .
$$
\noindent
Then, the Jacobian matrix of $F$ is 
\[
\begin{pmatrix}
0& 1 & 0 & 0 & 0 & 0 & 0 & 0 & 0 & 0  \\
1 & 0 & 0 & 0 & 0 & 0 & 0 & 0 & 0 & 0\\
0 & 0 & \frac{1}{\sqrt{2}} & \frac{1}{\sqrt{2}} & 0 & 0 & 0 & 0 & 0 & 0\\
0 & 0 & 0 & 0 & 0 & 0 & 0 & 0 & 0 & 0\\
0 & 0 & 0 & 0 & 0 & 0 & 0 & 0 & 0 & 0\\
0 & 0 & 0 & 0 & 0 & 0 & 0 & \frac{1}{\sqrt{2}}  & 0 & \frac{1}{\sqrt{2}} \\
0 & 0 & 0 & 0 & 0 & 0 & \frac{1}{\sqrt{2}}  & 0 & \frac{1}{\sqrt{2}}  & 0
\end{pmatrix}.
\]
Then, by direct calculations, we have
$$
\begin{aligned}
 {\operatorname{ker}} F_*=&\operatorname{span}\left\{V_1=\frac{\partial}{\partial x_3}-\frac{\partial}{\partial x_4}, V_2=\frac{\partial}{\partial x_8}-\frac{\partial}{\partial x_{10}},V_3=\frac{\partial}{\partial x_7}-\frac{\partial}{\partial x_9},V_4=\frac{\partial}{\partial x_5}, V_5=\frac{\partial}{\partial x_6}\right\}, \\
 \left(\operatorname{ker} F_*\right)^{\perp}=&\operatorname{span}\left\{X_1=\frac{\partial}{\partial x_1}, X_2=\frac{\partial}{\partial x_2},X_{3}=\frac{\partial}{\partial x_3}+\frac{\partial}{\partial x_4},X_{4}=\frac{\partial}{\partial x_8}+\frac{\partial}{\partial x_{10}},X_{5}=\frac{\partial}{\partial x_7}+\frac{\partial}{\partial x_9}\right\},\\
  {\operatorname{range} } F_*=&\operatorname{span}\left\{ \frac{\partial}{\partial y_1}, \frac{\partial}{\partial y_2}, \frac{\partial}{\partial y_3} , \frac{\partial}{\partial y_6} , \frac{\partial}{\partial y_7}   \right\},\\
 \left(\operatorname{range} F_*\right)^{\perp}=&\operatorname{span}\left\{ \frac{\partial}{\partial y_4}, \frac{\partial}{\partial y_5}  \right\},
\end{aligned}
$$
\noindent
where $\displaystyle\left\{\frac{\partial}{\partial x_1},\frac{\partial}{\partial x_2},\frac{\partial}{\partial x_3},\frac{\partial}{\partial x_4},\frac{\partial}{\partial x_5},\frac{\partial}{\partial x_6},\frac{\partial}{\partial x_7},\frac{\partial}{\partial x_8},\frac{\partial}{\partial x_9},\frac{\partial}{\partial x_{10}} \right\}$,$ \displaystyle\left\{ \frac{\partial}{\partial y_1},\frac{\partial}{\partial y_2},\frac{\partial}{\partial y_3},\frac{\partial}{\partial y_4},\frac{\partial}{\partial y_5},\frac{\partial}{\partial y_6},\frac{\partial}{\partial y_7} \right\}$ are bases on $T_{p} {\mathbb{R}}^{10}$  and $T_{F(p)} {\mathbb{R}}^{7}$, respectively, for all $p \in {\mathbb{R}}^{10}$. By direct computations, we can see that $F_*(X_1)=\dfrac{\partial}{\partial y_2}$, $F_*(X_2)=\dfrac{\partial}{\partial y_1}$,$F_*(X_3)=\dfrac{\partial}{\partial y_3}$,$F_*(X_4)=\dfrac{\partial}{\partial y_6}$,$F_*(X_5)=\dfrac{\partial}{\partial y_7}$. We know that $F$ is Riemannian map if and only if $g_{2}(F_{*}X_i,F_{*}X_j)= g_{1}(X_i,X_j) $ for $i,j=1,2,3,4,5$ and for all $X_i, X_j \in \left(\operatorname{ker} F_*\right)^{\perp} $. Thus, $F$ is a Riemannian map. Moreover, $J V_4=-V_5, J V_5=V_4, J V_3=-V_{2},JV_{2}=V_{3},J V_1=-X_3,$ therefore ${\mathcal{D}}_1=\operatorname{span}\left\{V_2, V_3, V_4,V_5\right\}$ and ${\mathcal{D}}_2=$ span $\left\{V_1\right\}$. Also, $J{\mathcal{D}}_1={\mathcal{D}}_1$ and ${\mathcal{D}}_2 \cap J{\mathcal{D}}_2=\{0\}$. \\ Thus, we can say that $F$ is a generic Riemannian map.\\

\noindent
Consider the Koszul formula for Levi-Civita connection $\nabla$ for ${\mathbb{R}}^{10}$
$$
2 g\left(\nabla_U V, W\right)=U g(V, W)+V g(W, U)-W g(U, V)-g([V, W], U)-g([U, W], V)+g([U, V], W)
$$
\noindent
for all $U, V, W \in {\mathbb{R}}^{10}$. By simple calculations, we obtain
$$
\nabla_{e_i} e_j=0 \text { for all } i, j=1, \ldots, 10.
$$
\noindent
Hence $T_U V=T_V U=T_U U=0$ for all $U, V \in \Gamma\left(\operatorname{ker} F_*\right)$. Therefore, the fibers of $F$ are totally geodesic. Thus, $F$ is Clairaut.
 \end{ex-new}
 
\begin{ex-new}
      Let $\left({\mathbb{R}}^{6},J, g_1\right)$ be a nearly K\"ahler manifold endowed with the Euclidean metric $g_1$ on ${\mathbb{R}}^{6}$ given by
$$
g_1=\sum_{i=1}^{6} d x_i^2
$$
\noindent
and canonical complex structure
$$
J\left(x_1, x_2, x_3, x_4, x_5, x_6\right)=\left( -x_2, x_1, -x_4, x_3, -x_6,x_5\right).
$$
The $J$-basis is $\left\{\left.e_i=\dfrac{\partial}{\partial x_i} \right\rvert\, i=1, \ldots, 6\right\}$. Let $\left({\mathbb{R}}^4, g_2\right)$ be a Riemannian manifold endowed with the metric $g_2=\displaystyle\sum_{i=1}^4 d y_i^2$.

Consider a map $F:\left({\mathbb{R}}^{6},J, g_1\right) \rightarrow\left({\mathbb{R}}^4, g_2\right)$ defined by

$$
F\left(x_1, x_2, x_3, x_4,x_5,x_6\right)=\left(\frac{x_3-x_4}{\sqrt{2}}, x_5, x_6, 0\right) .
$$
\noindent
Then, the Jacobian matrix of $F$ is 

\[
\begin{pmatrix}
0& 0 & \frac{1}{\sqrt{2}} &  -\frac{1}{\sqrt{2}}&0 & 0   \\
0 & 0 & 0 & 0 & 1 & 0 \\
0 & 0 & 0 & 0 & 0 &1\\
0 & 0 & 0 & 0 & 0 & 0 & \\
\end{pmatrix}.
\]
Then, by direct calculations, we have
$$
\begin{aligned}
 {\operatorname{ker} } F_*=&\operatorname{span}\left\{V_1=\frac{\partial}{\partial x_1}, V_2=\frac{\partial}{\partial x_2},V_3=\frac{\partial}{\partial x_3}+\frac{\partial}{\partial x_4}\right\}, \\
 \left(\operatorname{ker} F_*\right)^{\perp}=&\operatorname{span}\left\{X_1=\frac{\partial}{\partial x_3}-\frac{\partial}{\partial x_4}, X_2=\frac{\partial}{\partial x_5},X_{3}=\frac{\partial}{\partial x_6}\right\},\\
 {\operatorname{range} } F_*=&\operatorname{span}\left\{ \frac{\partial}{\partial y_1} , \frac{\partial}{\partial y_2} , \frac{\partial}{\partial y_3}  \right\},\\
 \left(\operatorname{range} F_*\right)^{\perp}=&\operatorname{span}\left\{ \frac{\partial}{\partial y_4}  \right\},
\end{aligned}
$$
\noindent
where $ \displaystyle\left\{\frac{\partial}{\partial x_1},\frac{\partial}{\partial x_2},\frac{\partial}{\partial x_3},\frac{\partial}{\partial x_4},\frac{\partial}{\partial x_5},\frac{\partial}{\partial x_6} \right\}$,$\displaystyle \left\{ \frac{\partial}{\partial y_1},\frac{\partial}{\partial y_2},\frac{\partial}{\partial y_3},\frac{\partial}{\partial y_4} \right\}$ are bases on $T_{p} {\mathbb{R}}^{6}$  and $T_{F(p)} {\mathbb{R}}^{4}$, respectively, for all $p \in {\mathbb{R}}^{6}$. By direct computations, we can see that $F_*(X_1)=\sqrt{2}\dfrac{\partial}{\partial y_1}$, $F_*(X_2)=\dfrac{\partial}{\partial y_2}$, $F_*(X_3)=\dfrac{\partial}{\partial y_3}$. We know that $F$ is a Riemannian map if and only if $g_{2}(F_{*}X_i,F_{*}X_j)= g_{1}(X_i,X_j) $ for $i,j=1,2,3$ and for all $X_i, X_j \in \left(\operatorname{ker} F_*\right)^{\perp} $. Thus, $F$ is a Riemannian map. Moreover, $J V_1=-V_2, J V_2=V_1, J V_3=X_1,JX_{1}=-V_3,J X_2 =-X_3,J X_3=X_2$; therefore, ${\mathcal{D}}_1=\operatorname{span}\left\{V_1, V_2\right\}$ and ${\mathcal{D}}_2=$ span $\left\{V_3\right\}$. Also,we have  $J{\mathcal{D}}_1={\mathcal{D}}_1$ and ${\mathcal{D}}_2 \cap J{\mathcal{D}}_2=\{0\}$. \\ Thus, $F$ is a generic Riemannian map.\\

\noindent
From the Koszul formula for the Levi-Civita connection $\nabla$ for ${\mathbb{R}}^6$ and simple calculations, we obtain
$$
\nabla_{e_i} e_j=0 \text { for all } i, j=1, \ldots, 6.
$$
\noindent
Hence $T_U V=T_V U=T_U U=0$ for all $U, V \in \Gamma\left(\operatorname{ker} F_*\right)$. Therefore, the fibers of $F$ are totally geodesic. Thus, $F$ is a Clairaut map.
 \end{ex-new}

\noindent{\bf Declaration:} We declare that no conflicts of interest are associated with this publication, and there has been no significant financial support for this work.  We certify that the submission is original work. The third author is supported and funded by the National Board of Higher Mathematics (NBHM) project no. 02011/21/2023NBHM(R.P.)/R\&DII/14960, INDIA.\\

\noindent{\bf Acknowledgment:}
The authors are deeply grateful to the reviewers for their insightful and constructive comments, which have significantly improved the quality of this work.

\end{document}